\documentclass[12pt]{article}
\usepackage[utf8]{inputenc}
\usepackage{amssymb, amsmath,amsthm, thmtools, mathabx, mathtools}
\usepackage{bbm}
\usepackage{amsfonts}
\usepackage[left=2.5cm,right=2.5cm,top=2cm,bottom=2cm]{geometry}
\usepackage{xcolor}

\usepackage{tabto}
\TabPositions{2 cm, 4 cm, 6 cm, 8 cm}

\usepackage{tikz-cd}

\usepackage{tikz}
\usetikzlibrary{arrows, automata, positioning}

\usepackage{tocloft}
\usepackage{appendix}

\usepackage{enumerate}

\usepackage{hyperref}
\usepackage[capitalize]{cleveref}

\hypersetup{
  colorlinks   = true, %Colours links instead of ugly boxes
  urlcolor     = blue, %Colour for external hyperlinks
  linkcolor    = blue, %Colour of internal links
  citecolor   = blue %Colour of citations
}
 
 \newtheorem{thmA}{Theorem}

\newtheorem{theorem}{Theorem}[section]

\newtheorem{lemma}[theorem]{Lemma}

\newtheorem*{theorem*}{Theorem}
\newtheorem*{lemma*}{Lemma}

\theoremstyle{definition}
\newtheorem{definition}[theorem]{Definition}

\newtheorem{remark}[theorem]{Remark}

\newtheorem{question}[theorem]{Question}

\newtheorem*{remark*}{Remark}
\newtheorem*{notation*}{Notation}
\newtheorem*{acks*}{Acknowledgements}
\newtheorem*{out*}{Outline}

\renewcommand\leq{\leqslant}
\renewcommand\geq{\geqslant}

\newcommand{\N}{\mathbb{N}}
\newcommand{\Z}{\mathbb{Z}}
\newcommand{\Q}{\mathbb{Q}}

\newcommand{\normal}[1]{\left<\! \left< #1\right> \!\right>}

\newcommand{\FP}{\mathtt{FP}}
\newcommand{\F}{\mathtt{F}}
\newcommand{\higig}{\mathrm{Hig}_\iota(G)}

\begin{document}

\title{Finiteness properties and Higman's rope trick}
\author{Francesco Fournier-Facio and Matthew C. B. Zaremsky}
\date{\today}
\maketitle

\begin{abstract}
In 1961 Higman proved that every finitely generated recursively presented group embeds into a finitely presented group. In 2018 Leary proved that every finitely generated group embeds into a group of type $\FP_2$. One naturally wonders whether analogous results hold for the higher finiteness properties $\F_n$ and $\FP_n$ ($3\le n\le \infty$), i.e., whether every finitely generated recursively presented group embeds into a group of type $\F_n$ and whether every finitely generated group embeds into a group of type $\FP_n$. We prove that a positive answer to the $\FP_n$ question that moreover preserves recursive presentability would imply a positive answer to the $\F_n$ question. We also investigate the output groups from Higman's and Leary's proofs, which in both cases arise from the so-called ``Higman rope trick'', and find that they are never of type $\FP_3(\Q)$ (hence also never $\FP_3$ nor $\F_3$). Thus, any approach to the questions of higher finiteness properties must go via a different route than the Higman rope trick.
\end{abstract}

\section{Introduction}

\begin{theorem}[The Higman embedding theorem \cite{higman}]
Every finitely generated recursively presented group embeds into a finitely presented group.
\end{theorem}

The Higman embedding theorem is one of the cornerstones of combinatorial group theory, which has found innumerable applications, such as the existence of a universal finitely presented group \cite{higman} and an algebraic characterisation of groups with solvable word problem \cite{BH}, just to name two early ones. The statement above actually implies that every \emph{countable} recursively presented group embeds into a finitely presented group, however the passage from countable to finitely generated is standard and easy \cite{HNN} (and the converse of the Higman embedding theorem is only true in the finitely generated case), so throughout this paper we only work with finitely generated groups for convenience.

\medskip

Much more recently, Leary proved the following surprising homological version of the Higman embedding theorem.

\begin{theorem}[\cite{leary:subgroups}]
    Every finitely generated group embeds into a group of type $\FP_2$.
\end{theorem}

Recall that a group $G$ is of \emph{type $\FP_n(A)$}, where $A$ is a commutative unital ring and $n \in \N \cup \{\infty\}$, if the trivial $A G$-module $A$ admits a projective resolution $P_* \to A$ such that $P_i$ is finitely generated for $0 \leq i \leq n$; we write $\FP_n = \FP_n(\Z)$, and note that $\FP_1$ is equivalent to being finitely generated. It was a longstanding open question whether every group of type $\FP_2$ is finitely presented---this was answered negatively by Bestvina and Brady \cite{BB}, and Leary's theorem is in a sense the most dramatic negative answer. 

\medskip

A fundamental open question (see \cite[Q 8.7]{bestvina}, \cite[21.146]{kourovka}, \cite[1.1]{matt}, and the comments to \cite{agol}) is whether the theorems of Higman and Leary admit higher-dimensional analogues. Recall that a group is of \emph{type $\F_n$} if it admits a classifying space with finitely many cells in each dimension up to $n$. Type $\F_1$ is equivalent to being finitely generated and type $\F_2$ is equivalent to being finitely presented. It is easy to see that type $\F_n$ implies type $\FP_n$, and in fact type $\FP_n$ together with finite presentability implies type $\F_n$ (Lemma \ref{lem:FP_to_F}).

\begin{question}\label{quest:higher_F}
    Does every finitely generated recursively presented group embed in a group of type $\F_3$? In a group of type $\F_\infty$?
\end{question}

\begin{question}\label{quest:higher_FP}
    Does every finitely generated group embed in a group of type $\FP_3$? In a group of type $\FP_\infty$?
\end{question}

Note that Question~\ref{quest:higher_FP} is not limited by cardinality issues, since there exist uncountably many groups of type $\FP_\infty$ \cite{leary:uncountably}. Also note that Question \ref{quest:higher_FP} is open even if one replaces type $\FP_n$ with the more permissive type $\FP_n(\Q)$ (to the best of our knowledge, there is no proof that every finitely generated group embeds into a group of type $\FP_2(\Q)$ simpler than the one in \cite{leary:subgroups} achieving $\FP_2$).

\medskip

Our first main result clarifies the relationship between a hypothetical higher Leary theorem and a hypothetical higher Higman theorem.

\begin{thmA}\label{thm:FP_to_F}
    Let $n \in \N \cup \{\infty\}$. If every finitely generated recursively presented group embeds into a recursively presented group of type $\FP_n$, then every finitely presented group embeds into a group of type $\F_n$.
\end{thmA}

Higman's original proof of the embedding theorem, as well as some subsequent accounts and simplifications of it (see e.g., \cite{rotman, LS}), rely crucially on the so-called \emph{Higman rope trick} \cite[Lemma IV.7.6]{LS}. Since the details will be important shortly, let us give a general statement and sketch a quick proof.

\begin{definition}[The group $\higig$]
\label{def:higig}
    Let $G$ be a finitely generated group, written as $F/R$, where $F$ is a finitely generated free group. Let $L = F_1 *_R F_2$ be the double, and let $\pi \colon L \to G$ be the surjection obtained by killing $F_2$. Let $\iota \colon L \to P$ be an embedding into a finitely generated group.

    We denote by $\higig$ the HNN extension with vertex group $P \times G$ and edge group $L$, conjugating the embeddings $\iota \times 1$ and $\iota \times \pi$.
\end{definition}

\begin{lemma}[Higman's rope trick]
    Let $G$ be a finitely generated group, written as $F/R$, where $F$ is a finitely generated free group. Let $\higig$ be as in Definition \ref{def:higig}, and suppose that $P$ is finitely presented. Then $\higig$ is finitely presented.
\end{lemma}

\begin{proof}[Proof sketch]
    Given $r \in R < L$, the element $(r, 1)$ is conjugate under the stable letter to $(r, r) \in P \times F/R$ (considering $r \in F_1)$ but also to $(r, 1)$ (considering $r \in F_2$). This shows that the finitely many relations given by those determining $P$, those saying the factors of the direct product commute, and those determining the conjugation in the HNN extension, imply for all $r\in R$ that $(r, r) = (r, 1)$ hence $(1, r) = 1$. Thus all the defining relations in $G$ are consequences of these finitely many, and $\higig$ is finitely presented.
\end{proof}

This reduces Higman's embedding theorem to proving that, if $R < F$ is a recursively enumerable normal subgroup, then the double $L = F_1 *_R F_2$ embeds into a finitely presented group, which is the core of Higman's proof. The proof of Leary's embedding theorem in \cite{leary:subgroups} also crucially employs the rope trick: one can show that in the above argument, if we only assume $P$ is $\FP_2$, then this is still enough to ensure that $\higig$ is $\FP_2$ \cite[Lemma 2.2]{leary:subgroups}. The name is due to Lyndon and Schupp \cite{LS}, in reference to the Indian rope trick. Schupp explains: ``As one is concentrating on the details of the proof, Higman the magician makes the infinitely many relations disappear'' \cite{name}.

\medskip

Our second main result is that answering Question~\ref{quest:higher_F} or~\ref{quest:higher_FP} would necessarily have to go through a different route than the rope trick. We emphasise that in the below theorem $P$ and $G$ could have as strong of finiteness properties as desired, they could even be of type $\F_\infty$.

\begin{thmA}
\label{thm:rope}
    Let $G = F/R$, and let $\higig$ be as in Definition \ref{def:higig}. Suppose $G$ is infinite and $R \neq 1$. Then $H_3(\higig; \Q)$ is infinite-dimensional; in particular $\higig$ is not of type $\FP_3(\Q)$.
\end{thmA}

Theorem \ref{thm:rope} shows that an answer to Question \ref{quest:higher_F} or \ref{quest:higher_FP} would either need some sort of higher version of the rope trick, or perhaps go through the alternative, and much less elementary, approach of $S$-machines, introduced in \cite{S}. These techniques can produce Higman embeddings that are drastically different from the classical ones and give more control; for example one can even ensure that the recursively presented subgroup is malnormal in the finitely presented one \cite{wagner}, something one cannot obtain with tricks involving direct products.

Moreover, Theorem \ref{thm:rope} combined with the Higman embedding theorem provides a rich source of groups that are finitely presented but not of type $\FP_3(\Q)$. There are by now a wealth of examples with these properties, but they are still quite hard to come by. The first example, by Stallings \cite{stallings}, is the kernel of a map from a product of free groups to $\Z$; later Brady provided an example as the kernel of a map from a hyperbolic group to $\Z$ \cite{brady}, and this is still the most common source of groups with intermediate finiteness properties \cite{LIP}. Other classes of examples arise as arithmetic \cite{arithmetic} or Houghton groups \cite{brown87}, and these can in turn be used to produce simple groups with intermediate finiteness properties \cite{SWZ, tbt}. We find it remarkable that the groups in Theorem~\ref{thm:rope} providing these ``new'' examples were already present in Higman's work \cite{higman}, predating the proven first example of Stallings \cite{stallings}.

\begin{acks*}
    We thank Ian Leary for helpful discussions. FFF is supported by the Herchel Smith Postdoctoral Fellowship Fund.
\end{acks*}

\section{From homological to homotopical}\label{sec:hlgy_to_htpy}

Perhaps the most fundamental fact involving homological and homotopical finiteness properties is that type $\F_2$ plus type $\FP_n$ implies type $\F_n$. This is usually attributed to Wall, after \cite{wall1, wall2}, where similar statements appear. Despite its fundamental importance, we were surprised that we could not find an explicit proof of exactly this fact in the literature. Since this plays an essential role in the proof of Theorem \ref{thm:FP_to_F}, for the sake of completeness and expository clarity, we include a proof, which is analogous to the proof of the Eilenberg--Ganea theorem \cite{EG} (see also \cite[Section VIII.7]{brown}).

\begin{lemma}\label{lem:FP_to_F}
    If a finitely presented group $G$ is of type $\FP_n$, then it is of type $\F_n$.
\end{lemma}

\begin{proof}
    If $n\le 2$ there is nothing to prove, so assume $n\ge 3$. By induction we may assume $G$ is of type $\F_{n-1}$ and $\FP_n$, and must prove it is of type $\F_n$. Since $G$ has a classifying space with finite $(n-1)$-skeleton, looking at the $(n-1)$-skeleton of the universal cover we get an $(n-1)$-dimensional free cocompact $G$-complex $X$ with vanishing $\pi_k$ for all $k\le (n-2)$. Our goal is to attach finitely many orbits of $n$-cells to $X$, to obtain an $n$-dimensional free cocompact $G$-complex $Y$ with vanishing $\pi_k$ for all $k \le (n-1)$. Assuming this for the moment, the space $G \backslash Y$ is a finite $n$-dimensional CW-complex with fundamental group $G$, which by the long exact sequence of a fibration (recall that we are assuming $n \geq 3$) has vanishing $\pi_k$ for all $2\leq k \leq (n-1)$. We can then attach $k$-cells to $G \backslash Y$ inductively for all $k > n$ to obtain a classifying space for $G$ with the same (finite) $n$-skeleton as $G \backslash Y$, revealing that $G$ is of type~$\F_n$. 
    
    It remains to construct $Y$. The action of $G$ on $X$ gives the cellular chain complex $C_*(X)$ the structure of a resolution of $\Z$ by free $\Z G$-modules. In particular
    \[H_{n-1}(X) \to C_{n-1}(X) \to C_{n-2}(X) \to \cdots \to C_0(X) \to \Z\]
    is an exact sequence of $\Z G$-modules, which are finitely generated with the possible exception of the leftmost term. Since $G$ is of type $\FP_n$, there is also an exact sequence of finitely generated projective $\Z G$-modules
    \[P_n \to P_{n-1} \to \cdots \to P_0 \to \Z.\]
    By the fundamental theorem of homological algebra, the identity on $\Z$ extends uniquely to a $G$-equivariant chain map $f_* \colon P_* \to C_*(X)$ (with $C_n(X)$ replaced by $H_{n-1}(X)$), which implies that $H_{n-1}(X)$ is finitely generated as a $\Z G$-module.
    
    Since $X$ is $(n-2)$-connected, by the Hurewicz theorem each of the finitely many $\Z G$-generators of $H_{n-1}(X) \cong \pi_{n-1}(X)$ is represented by a map $S^{n-1} \to X$. For each such generator we can extend the map from $S^{n-1}$ to a map from the $n$-disk $D^n$ by adding an $n$-cell to $X$. When we do this for each $\Z G$ generator, and extend $G$-equivariantly, we obtain an $n$-dimensional free cocompact $G$-complex $Y$ with $(n-1)$-skeleton $X$ (hence still vanishing $\pi_k$ for $k\le n-2$) and now with vanishing $\pi_{n-1}(X)$, which concludes the proof.
\end{proof}

Now we can prove Theorem~\ref{thm:FP_to_F}. The construction is inspired by the proof of \cite[Theorem~E]{BDM}, which states that every recursively presented group embeds into a finitely presented acyclic group. The same idea was later used in \cite[Theorem 4]{FFLM} to give the first example of a finitely presented non-amenable boundedly acyclic group. We remark that the essence of the proof is another kind of ``rope trick''.

\begin{proof}[Proof of Theorem~\ref{thm:FP_to_F}]
    Suppose that every finitely generated recursively presented group embeds into a recursively presented group of type $\FP_n$. We need to show that every finitely presented group $G$ embeds into a group of type $\F_n$, and by \cite{higman} we may assume that $G$ is a universal finitely presented group. Let $\iota \colon G \to H$ be an embedding into a recursively presented group $H$ of type $\FP_n$. By the Higman embedding theorem, there is an embedding $\jmath \colon H \to G$. We have a chain of inclusions $G \xrightarrow{\iota} H \xrightarrow{\jmath} G$.
    The composition $\jmath \iota$ is a self-embedding of $G$; let $E$ be the corresponding ascending HNN extension, with stable letter $t$. Being an ascending HNN extension of $G$, which is finitely presented, $E$ is finitely presented.

    We now claim that $E$ is isomorphic to an ascending HNN extension of $\jmath(H) \cong H$. By \cite[Lemma 3.1]{ascending}, we need to check three things.
    \begin{enumerate}
        \item $t\jmath(H)t^{-1} \subset \jmath(H)$. Indeed $t \jmath(H) t^{-1} \subset tGt^{-1} = \jmath \iota(G) \subset \jmath(H)$.
        \item $E = \langle \jmath(H), t \rangle$. Indeed $\jmath(H)$ contains $\jmath \iota(G) = t^{-1} G t$ and $E = \langle G, t \rangle$.
        \item $t^n \notin \jmath(H)$ for any $n \neq 0$, This follows from the fact that $\jmath(H) \subset G$.
    \end{enumerate}
    Since $H$ is of type $\FP_n$, by \cite[Proposition 2.12]{bieri}, its ascending HNN extension $E$ is also of type $\FP_n$. Thus $E$, being both finitely presented and of type $\FP_n$, is of type $\F_n$ by Lemma~\ref{lem:FP_to_F}.
\end{proof}

\begin{remark}
    Even though we do not have a higher version of Leary's theorem, it is very reasonable to expect that such a version would preserve recursive presentability, which amounts to saying that the embedding should be more or less explicit. This is the case for Leary's embedding in the $\FP_2$ case \cite{leary:subgroups}, for instance.
    
    Indeed, start with a group $G = F/R$. The first step \cite[Lemma 2.5]{leary:subgroups} embeds the double $L = F *_R F$ into a group $P$ of type $\FP_2$. This uses a construction from \cite{leary:uncountably}, which is applied in \cite[Lemma 2.4]{leary:subgroups}, identifying a specific subgroup $V_S$ of a finitely generated free group. The construction takes as input a set $S$ of natural numbers, and in the relevant application this set is some coding of the relations in $R$. Hence if $R$ is recursively enumerable, $V_S$ is also recursively enumerable, and all of the following steps of the construction are explicit, so that $P$ is recursively presented.

    The second step is the Higman rope trick \cite[Lemma 2.2]{leary:subgroups}: the type $\FP_2$ container for $G$ is the HNN extension $\higig$ with vertex group $P \times G$ and edge group $L$ conjugating the embeddings $\iota \times 1$ and $\iota \times \pi$. Since $\higig$ is obtained from $P \times G$ by adding a generator and finitely many relations, and since $P \times G$ is recursively presented, we conclude that $\higig$ is too.
\end{remark}

\section{Finiteness properties of Higman containers}\label{sec:main}

Now we move on to the proof of Theorem~\ref{thm:rope}. We keep the notation from Definition \ref{def:higig}: $G = F/R$, $L = F_1 *_R F_2$, $\iota \colon L \to P$, and $\pi \colon L \to G$, and $\higig$ is the HNN extension conjugating the embeddings $\iota \times 1, \iota \times \pi \colon L \to P \times G$. We always assume that $F$ (hence $G$ and $L$) and $P$ (hence $\higig$) are finitely generated.

\begin{remark}\label{rem:assumptions}
    If $G$ is finite, then $R$ is a finite-index subgroup of $F$, hence $L$ is of type $\F_\infty$. Similarly, if $R = 1$, then $L$ is a free group, hence it is of type $\F_\infty$. It then follows from \cite[Proposition 2.13]{bieri} and Lemma \ref{lem:FP_to_F} that $\higig$ will inherit the finiteness properties of $P \times G$, hence the finiteness properties of $P$. This is why the statement requires that $G$ is infinite and $R \neq 1$, which we assume from now on. Note that these assumptions are equivalent to saying that the free group $R$ is infinite rank.
\end{remark}

\begin{lemma}
\label{reduction to ker L to G}
    If $\ker(H_2(L; \Q) \xrightarrow{H_2(\pi)} H_2(G; \Q))$ is infinite-dimensional, then $H_3(\higig; \Q)$ is infinite-dimensional.
\end{lemma}

\begin{proof}
    We omit the rational coefficients throughout the proof. The Mayer--Vietoris sequence for HNN extensions \cite[Theorem 2.12]{bieri} includes
    \[H_3(\higig) \to H_2(L) \xrightarrow{\alpha} H_2(P \times G),\]
    where $\alpha = H_2(\iota \times 1) - H_2(\iota \times \pi)$. By the K{\"u}nneth formula, we have a surjection
    \[H_2(P \times G) \cong H_2(P) \oplus \left( H_1(P) \otimes H_1(G) \right) \oplus H_2(G) \xrightarrow{\kappa} H_2(P) \oplus H_2(G),\]
    by killing the mixed term. Since both $G$ and $P$ are finitely generated, $\ker(\kappa)$ is finite-dimensional. The composition
    \[H_2(L) \xrightarrow{\kappa \circ \alpha} H_2(P) \oplus H_2(G)\]
    is equal to
    \[(H_2(\iota) \oplus 0) - (H_2(\iota) \oplus H_2(\pi)) = 0 \oplus - H_2(\pi).\]
    Since we are assuming that $\ker(H_2(\pi))$ is infinite-dimensional, we deduce that $\ker(\kappa \circ \alpha)$ is infinite-dimensional. But $\ker(\kappa)$ is finite-dimensional, so $\ker(\alpha)$ must be infinite-dimensional. Finally, $H_3(\higig; \Q)$ surjects onto $\ker(\alpha)$, and we conclude.
\end{proof}

\begin{lemma}
\label{L inf dim}
    We have an isomorphism $H_2(L; \Z) \cong \ker(H_1(R; \Z) \to H_1(F; \Z))$. In particular, $H_2(L; \Q)$ is infinite-dimensional.
\end{lemma}

\begin{proof}
    The Mayer--Vietoris sequence for amalgams \cite[Theorem 2.10]{bieri} includes
    \[H_2(F; \Z)^2 \to H_2(L; \Z) \to H_1(R; \Z) \to H_1(F; \Z)^2.\]
    The first term vanishes because $F$ is free. The last arrow is the map induced by the inclusion $R \to F$ in the first coordinate, and its negative on the second coordinate, so it has the same kernel as the map induced by the inclusion itself, proving the first statement. The last statement follows from the fact that $H_1(F;\Z)$ is finitely generated, while $H_1(R; \Z)$ is a free abelian group of infinite rank, hence by the universal coefficient theorem $H_2(L;\Q)$ is the rationalisation of a free abelian group of infinite rank.
\end{proof}

This is already enough to prove Theorem \ref{thm:rope} in case $G$ is of type $\FP_2$, or simply $H_2(G; \Q)$ is finite-dimensional. In general, we are left with proving:

\begin{lemma}
\label{ker inf dim}
    The induced map $H_2(\pi) \colon H_2(L; \Q) \to H_2(G; \Q)$ has infinite-dimensional kernel.
\end{lemma}

\begin{proof}
    By the universal coefficient theorem it suffices to show that the kernel of the induced map in integral homology has infinite rank.
    
    Let us fix some notation: let $F_1, F_2$ be two copies of $F$, and $R_1, R_2$ the corresponding copies of $R$. Then $L = (F_1 * F_2) / N$, where $N = \normal{r_1 r_2^{-1} \mid r \in R}$, and the map $\pi \colon L \to G \cong F_1/R_1$ is given by killing all of $F_2$. The natural Hopf formula \cite[II.5]{brown} realises the induced map in $H_2(- ; \Z)$ as:
    \[\frac{N \cap [F_1 *F_2, F_1 * F_2]}{[F_1 *F_2, N]} \xrightarrow{r_1r_2^{-1} \mapsto r} \frac{R \cap [F, F]}{[F, R]}.\]
    The connecting homomorphism $H_2(L; \Z) \to \ker(H_1(R; \Z) \to H_1(L; \Z))$, which is an isomorphism by Lemma \ref{L inf dim}, can be identified as follows, using the Hopf formula, and the identification of $H_1(-;\Z)$ with the abelianisation:
    \[\frac{N \cap [F_1 *F_2, F_1 * F_2]}{[F_1 *F_2, N]} \xrightarrow{r_1r_2^{-1} \mapsto r} \ker\left(\frac{R}{[R, R]} \to \frac{F}{[F, F]}\right) = \frac{R \cap [F, F]}{[R, R]}.\]
    Putting together these maps identifies the kernel of $H_2(\pi)$ with the abelian group $[F, R]/[R, R]$. This is a subgroup of the free abelian group $R/[R, R]$, so it remains to show that it has infinite rank. This is easiest to see topologically.
    
    Let $X$ be a wedge of circles with fundamental group $F$, and let $Y \to X$ be the regular covering space associated to $R$. Hence $R = \pi_1(Y)$ and the deck transformation group is $G$. The homology $H_1(Y)$ is naturally identified with $H_1(R) = R/[R, R]$. Under this identification, the group $[F, R]/[R, R]$ is identified with the subgroup $I_G \cdot H_1(Y) = \langle gc - c \mid g \in G, c \in H_1(Y) \rangle$. Now pick an embedded circuit $c$ in $Y$, which we see as an element of $H_1(Y)$. Because $X$ is compact and $Y$ is not, there is a sequence $(g_n)_{n \geq 1}$ of elements in $G$ such that the circuits $(g_n c)_{n \geq 1}$ are pairwise disjoint, and hence the subgroup $\langle g_n c \mid n \geq 1\rangle < H_1(Y)$ has infinite rank. But this is contained in $\langle c, g_n c - c \mid n \geq 1\rangle < \langle c, I_G \cdot H_1(Y) \rangle$, and so $I_G \cdot H_1(Y)$ must also have infinite rank.
\end{proof}

\begin{proof}[Proof of Theorem \ref{thm:rope}]
    Combine Lemmas \ref{reduction to ker L to G} and \ref{ker inf dim}.
\end{proof}

\begin{remark}
    It is important in all of this that $F$ is a free group. Thus one might wonder whether viewing $G$ as $F/R$ for some non-free $F$ could lead to a higher version of the rope trick. Unfortunately the other part of the proof of these embedding theorems, for example Leary's proof that $\langle F,t\mid tr=rt$ for all $r\in R\rangle$ always embeds in a group of type $\FP_2$, relies heavily on $F$ being free.
\end{remark}

\footnotesize

\bibliographystyle{amsalpha}
\bibliography{ref}

\providecommand{\bysame}{\leavevmode\hbox to3em{\hrulefill}\thinspace}
\providecommand{\MR}{\relax\ifhmode\unskip\space\fi MR }
\providecommand{\MRhref}[2]{%
  \href{http://www.ams.org/mathscinet-getitem?mr=#1}{#2}
}
\providecommand{\href}[2]{#2}
\begin{thebibliography}{GMSW01}

\bibitem[Ago]{agol}
Ian Agol, \emph{A torsionfree group with infinite cohomological dimension and
  no infinitely generated free abelian subgroup}, MathOverflow,
  \url{https://mathoverflow.net/q/60565}.

\bibitem[BB97]{BB}
Mladen Bestvina and Noel Brady, \emph{Morse theory and finiteness properties of
  groups}, Invent. Math. \textbf{129} (1997), no.~3, 445--470. \MR{1465330}

\bibitem[BDM83]{BDM}
G.~Baumslag, E.~Dyer, and C.~F. Miller, III, \emph{On the integral homology of
  finitely presented groups}, Topology \textbf{22} (1983), no.~1, 27--46.
  \MR{682058}

\bibitem[Bes]{bestvina}
Mladen Bestvina, \emph{Questions in geometric group theory},
  \url{https://www.math.utah.edu/~bestvina/eprints/questions-updated.pdf}.

\bibitem[BH74]{BH}
William~W. Boone and Graham Higman, \emph{An algebraic characterization of
  groups with soluble word problem}, J. Austral. Math. Soc. \textbf{18} (1974),
  41--53. \MR{357625}

\bibitem[Bie81]{bieri}
Robert Bieri, \emph{Homological dimension of discrete groups}, second ed.,
  Queen Mary College Mathematics Notes, Queen Mary College, Department of Pure
  Mathematics, London, 1981. \MR{715779}

\bibitem[Bra99]{brady}
Noel Brady, \emph{Branched coverings of cubical complexes and subgroups of
  hyperbolic groups}, J. London Math. Soc. (2) \textbf{60} (1999), no.~2,
  461--480. \MR{1724853}

\bibitem[Bro87]{brown87}
Kenneth~S. Brown, \emph{Finiteness properties of groups}, Proceedings of the
  {N}orthwestern conference on cohomology of groups ({E}vanston, {I}ll., 1985),
  vol.~44, 1987, pp.~45--75. \MR{885095}

\bibitem[Bro94]{brown}
\bysame, \emph{Cohomology of groups}, Graduate Texts in Mathematics, vol.~87,
  Springer-Verlag, New York, 1994, Corrected reprint of the 1982 original.
  \MR{1324339}

\bibitem[BW07]{arithmetic}
Kai-Uwe Bux and Kevin Wortman, \emph{Finiteness properties of arithmetic groups
  over function fields}, Invent. Math. \textbf{167} (2007), no.~2, 355--378.
  \MR{2270455}

\bibitem[BZ22]{tbt}
James Belk and Matthew C.~B. Zaremsky, \emph{Twisted {B}rin-{T}hompson groups},
  Geom. Topol. \textbf{26} (2022), no.~3, 1189--1223. \MR{4466647}

\bibitem[EG57]{EG}
Samuel Eilenberg and Tudor Ganea, \emph{On the {L}usternik-{S}chnirelmann
  category of abstract groups}, Ann. of Math. (2) \textbf{65} (1957), 517--518.
  \MR{85510}

\bibitem[FFLM24]{FFLM}
Francesco Fournier-Facio, Clara L{\"o}h, and Marco Moraschini, \emph{Bounded
  cohomology of finitely presented groups: vanishing, non-vanishing, and
  computability}, Ann. Sc. Norm. Super. Pisa Cl. Sci. (5) \textbf{25} (2024),
  no.~2, 1169--1202. \MR{4778473}

\bibitem[GMSW01]{ascending}
Ross Geoghegan, Michael~L. Mihalik, Mark Sapir, and Daniel~T. Wise,
  \emph{Ascending {HNN} extensions of finitely generated free groups are
  {H}opfian}, Bull. London Math. Soc. \textbf{33} (2001), no.~3, 292--298.
  \MR{1817768}

\bibitem[Hig61]{higman}
Graham Higman, \emph{Subgroups of finitely presented groups}, Proc. Roy. Soc.
  London Ser. A \textbf{262} (1961), 455--475. \MR{130286}

\bibitem[HNN49]{HNN}
Graham Higman, B.~H. Neumann, and Hanna Neumann, \emph{Embedding theorems for
  groups}, J. London Math. Soc. \textbf{24} (1949), 247--254. \MR{32641}

\bibitem[KM]{kourovka}
E.~I. Khukhro and V.~D. Mazurov, \emph{The {K}ourovka notebook},
  \url{https://kourovkanotebookorg.wordpress.com/wp-content/uploads/2026/07/21tkt.pdf}.

\bibitem[Lea18a]{leary:subgroups}
Ian~J. Leary, \emph{Subgroups of almost finitely presented groups}, Math. Ann.
  \textbf{372} (2018), no.~3-4, 1383--1391. \MR{3880301}

\bibitem[Lea18b]{leary:uncountably}
\bysame, \emph{Uncountably many groups of type {$FP$}}, Proc. Lond. Math. Soc.
  (3) \textbf{117} (2018), no.~2, 246--276. \MR{3851323}

\bibitem[LIP24]{LIP}
Claudio Llosa~Isenrich and Pierre Py, \emph{Subgroups of hyperbolic groups,
  finiteness properties and complex hyperbolic lattices}, Invent. Math.
  \textbf{235} (2024), no.~1, 233--254. \MR{4688705}

\bibitem[LS01]{LS}
Roger~C. Lyndon and Paul~E. Schupp, \emph{Combinatorial group theory}, 1977
  ed., Classics in Mathematics, Springer-Verlag, Berlin, 2001. \MR{1812024}

\bibitem[Rot95]{rotman}
Joseph~J. Rotman, \emph{An introduction to the theory of groups}, fourth ed.,
  Graduate Texts in Mathematics, vol. 148, Springer-Verlag, New York, 1995.
  \MR{1307623}

\bibitem[SBR02]{S}
Mark~V. Sapir, Jean-Camille Birget, and Eliyahu Rips, \emph{Isoperimetric and
  isodiametric functions of groups}, Ann. of Math. (2) \textbf{156} (2002),
  no.~2, 345--466. \MR{1933723}

\bibitem[Shl]{name}
ShlomiF, \emph{Why is ``the {H}igman rope trick'' thus named?}, MathOverflow,
  \url{https://mathoverflow.net/q/195126}.

\bibitem[Sta63]{stallings}
John Stallings, \emph{A finitely presented group whose 3-dimensional integral
  homology is not finitely generated}, Amer. J. Math. \textbf{85} (1963),
  541--543. \MR{158917}

\bibitem[SWZ19]{SWZ}
Rachel Skipper, Stefan Witzel, and Matthew C.~B. Zaremsky, \emph{Simple groups
  separated by finiteness properties}, Invent. Math. \textbf{215} (2019),
  no.~2, 713--740. \MR{3910073}

\bibitem[Wag]{wagner}
Francis Wagner, \emph{Malnormal subgroups of finitely presented groups}, arXiv
  preprint arXiv:2404.00841.

\bibitem[Wal65]{wall1}
C.~T.~C. Wall, \emph{Finiteness conditions for {${\rm CW}$}-complexes}, Ann. of
  Math. (2) \textbf{81} (1965), 56--69. \MR{171284}

\bibitem[Wal66]{wall2}
\bysame, \emph{Finiteness conditions for {${\rm CW}$} complexes. {II}}, Proc.
  Roy. Soc. London Ser. A \textbf{295} (1966), 129--139. \MR{211402}

\bibitem[Zar]{matt}
Matthew C.~B. Zaremsky, \emph{Some open problems},
  \url{https://zaremsky.github.io/open_problems.pdf}.

\end{thebibliography}

\vspace{0.5cm}

\normalsize

\noindent{\textsc{Department of Pure Mathematics and Mathematical Statistics, University of Cambridge, UK}}

\noindent{\textit{E-mail address:} \texttt{ff373@cam.ac.uk}} \\

\noindent{\textsc{Department of Mathematics and Statistics, University at Albany (SUNY), Albany, NY, USA}}

\noindent{\textit{E-mail address:} \texttt{mzaremsky@albany.edu}}

\end{document}